\documentclass[12pt, landscape]{amsart}

\usepackage{tikz}
\usepackage{mathtools, amssymb}
\usepackage{bbm}
\usepackage{doi}
\usepackage{yfonts}
\usepackage[textsize=footnotesize,
color=yellow!50]{todonotes}
\usepackage[margin=6cc]{geometry}

\DeclareMathSizes{10}{10}{8}{7}
\DeclareMathSizes{11}{11}{9}{8}
\DeclareMathSizes{12}{12}{10}{9}

\newtheorem{theorem}{Theorem}[section]
\newtheorem{lemma}[theorem]{Lemma}
\newtheorem{proposition}[theorem]{Proposition}
\newtheorem{corollary}[theorem]{Corollary}

\newtheorem{fact}[theorem]{Fact}

\theoremstyle{definition}
\newtheorem{definition}[theorem]{Definition}
\newtheorem{question}[theorem]{Question}

\theoremstyle{remark}

\newtheorem{claim*}[theorem]{Claim}

\newenvironment{proofofclaim}{{\itshape Proof of Claim.}}{\hfill\copyright}

\newcommand{\cfi}{\mathrm{cf}}

\newcommand{\soast}[2]{\mathord \{#1 \mid #2\}}
\newcommand{\bigsoast}[2]{\mathord \big\{#1 \bigm\vert #2\big\}}
\newcommand{\Bigsoast}[2]{\mathord \Big\{#1 \Bigm\vert #2\Big\}}
\newcommand{\soafft}[2]{{}^{#1} #2}

\newcommand{\seq}[2]{\langle #1 \mid #2\rangle}

\newcommand{\cf}{cf.}

\newcommand{\ie}{i.e.}

\newcommand{\card}[1]{\mathord |#1|}
\newcommand{\pwim}[2]{
#1\big[#2\big]
}
\newcommand{\ifff}{\text{ if and only if }}
\newcommand{\dfeq}{:=}
\newcommand{\wlg}{w.l.o.g.}

\newcommand{\disj}{\mathrel{\scriptstyle\vee}}

\newcommand{\doesntarrow}{{\mspace{7mu}\not\mspace{-7mu}\longrightarrow}}
\newcommand{\squarebinom}[2]{\biggl[\begin{matrix} #1 \\ #2 \end{matrix}\biggr]}
\newcommand{\sqtarg}[4]{
\begingroup
\renewcommand*{\arraystretch}{0.4}
\arraycolsep=2dd
\biggl[\begin{matrix} #1 && #3 \\ & \disj \\ #2 && #4 \end{matrix}\biggr]
\endgroup
}

\newcommand{\bigpolar}[2]{
\begingroup
\renewcommand*{\arraystretch}{0.4}
\arraycolsep=2dd
\biggl(\begin{matrix} #1 & \\ & \,:\, i<m\\ #2 & \end{matrix}\biggr)
\endgroup
}

\newcommand{\stick}{{\ensuremath \mspace{2mu}\mid\mspace{-12mu} {\raise0.6em\hbox{$\bullet$}}}}
\newcommand{\whitestick}[1]{{\ensuremath \mspace{2mu}\mid\mspace{-16.5mu} {\raise.95em\hbox{$\bigcirc \mspace{-13mu} \raise.05em\hbox{$ \scriptscriptstyle #1$}$}}}}

\newcommand{\calabash}{{\ensuremath \mspace{2mu}\mid\mspace{-12mu} {\raise-0.6em\hbox{$\bullet$}}}}
\newcommand{\match}{{\ensuremath \mspace{2mu}\mid\mspace{-3mu}\mid\mspace{-12.9mu} {\raise0.6em\hbox{$\bullet$}}}}
\newcommand{\matchop}{{\ensuremath \mspace{2mu}\mid\mspace{-3mu}\mid\mspace{-12.9mu} {\raise-0.6em\hbox{$\bullet$}}}}

\newcommand{\opair}[2]{
\mathord \langle #1, #2\rangle
}

\newcommand{\iz}{is}
\newcommand{\ou}{ou}
\newcommand{\re}{re}

\DeclareMathOperator{\cof}{\mathsf{cof}}
\DeclareMathOperator{\cov}{\mathsf{cov}}
\DeclareMathOperator{\add}{\mathsf{add}}
\DeclareMathOperator{\non}{\mathsf{non}}

\DeclareMathOperator{\meagre}{\mathcal{M}}
\DeclareMathOperator{\mzero}{\mathcal{N}}
\DeclareMathOperator{\MA}{\mathsf{MA}}
\DeclareMathOperator{\CH}{\mathsf{CH}}
\DeclareMathOperator{\GCH}{\mathsf{GCH}}
\DeclareMathOperator{\ZFC}{\mathsf{ZFC}}

\newcount\skewfactor
\def\mathunderaccent#1#2 {\let\theaccent#1\skewfactor#2
\mathpalette\putaccentunder}
\def\putaccentunder#1#2{\oalign{$#1#2$\crcr\hidewidth
\vbox to.2ex{\hbox{$#1\skew\skewfactor\theaccent{}$}\vss}\hidewidth}}
\def\name{\mathunderaccent\tilde-3 }

\begin{document}
\title{Cardinal Characteristics of the Continuum and Partitions}
\author{William Chen}
\address{Department of Mathematics, Ben-Gurion-University of the Negev \\
Beersheba, 8410501, Isra\"el}
\email{chenwb@gmail.com}
\author{Shimon Garti}
\address{Institute of Mathematics, The Hebrew University of Jerusalem \\
Jerusalem, 91904, Isra\"el}
\email{shimon.garty@mail.huji.ac.il}
\author{Thilo Weinert}
\address{DMG/Algebra, TU Wien \\
            Wiedner Hauptstrasse 8-10/104 \\
            1040 Wien \\
            Autriche}
\email{aleph\_omega@posteo.net}
\thanks{A portion of this research was undertaken while both the first and the last author were postdoctoral fellows at the Ben-Gurion-University of the Negev. They would like to thank the Ben-Gurion University of the Negev and the Israel Science Foundation which supported this research(grant \#1365/14).
}

\begin{abstract}
We prove that for regular cardinals $\kappa$, combinations of the stick principle at $\kappa$ and certain cardinal characteristics at $\kappa$ being $\kappa^+$ causes the partition relations such as $\omega_1 \longrightarrow (\omega_1, \omega + 2)^2$ and $(\kappa^+)^2 \longrightarrow (\kappa^+\kappa, 4)^2$ to fail. Polar\iz ed partition relations are also considered, and the results are used to answer several problems posed by Garti, Larson and Shelah.
\end{abstract}
\maketitle

\section{Introduction}
\label{section : introduction}

In this paper we consider the effect of cardinal characteristics of the continuum on partition relations. Connections between these two areas have been investigated before, a recent contribution is \cite{014RT0}. For a general overview of cardinal characteristics, \cf\ \cite{010B0, 012H1}, for partition relations \cf\ \cite{977W0, 984EHMR0, 011HL0}.

For an ordered set $X$ and an ordinal $\alpha$, let $[X]^{\alpha}$ denote the subsets of $X$ of order-type $\alpha$. In Rado's notation, \cf\ \cite{956ER0}, given ordinal parameters,
\begin{align}
\alpha\longrightarrow (\beta_i:i<m)^k_{m}
\end{align}
is the statement that for any $f:[\alpha]^k\rightarrow m$, there are $i<m$ and $X\in [\alpha]^{\beta_i}$ for which $f\upharpoonright [X]^k$ is constant with value $i$ (we say that $X$ is \textit{homogeneous} for $i$).

We will also consider polar\iz ed partition relations. There,
\begin{align}
\binom{\alpha}{\beta} \longrightarrow \bigpolar{\gamma_i}{\delta_i}^{1,1}_m
\end{align}
means that for every $f:\alpha\times\beta\rightarrow m$, there are $i<m$,  $X\in [\alpha]^{\gamma_i}$, and $Y\in [\beta]^{\delta_i}$ so that $f\upharpoonright [X\times Y]$ is constant with value $i$.

The notation is flexible in several ways. Replacing $\longrightarrow$ with $\doesntarrow$ gives the negation of the original statement. If there is $\beta$ so that $\beta_i=\beta$ for all $i<m$, then we will drop the indexing on the $\beta_i$. In this case, if the number of col\ou rs is $2$, we often omit the lower subscript. Replacing the parentheses on the right-hand side by square brackets weakens the statement so that we do not require $X$ to be homogeneous but just so that $f[X]\neq m$ (i.e., $X$ misses one col\ou r, instead of missing all col\ou rs but one, \cf\ \cite[Definition 18.1]{965EHR0}).

Finally, we can replace $\beta_i$ by $(\beta_i:\gamma_i)$. This notation means that instead of asking for a homogeneous set for $i$ of order-type $\beta_i$, we ask for $X$ of order-type $\beta_i$ and $Y$ of order-type $\gamma_i$ so that $x<y$ and $f(x,y)=i$ for every $x\in X$ and $y\in Y$. This is weaker than requiring a homogeneous set for $i$ of order-type $\beta_i+\gamma_i$.

In early work in this area, many negative partition relations involving ordinals $\leqslant \omega_1$ were shown to follow from the assumption of $\CH$. Subsequently, this assumption was reduced in many cases to $\mathfrak{cc}=\aleph_1$, where $\mathfrak{cc}$ is a cardinal characteristic. This gives a way of calibrating more precisely how much of $\CH$ is actually needed to prove a negative partition relation.

For a cardinal $\kappa$, the following cardinal characteristics will be used in this paper:
\begin{enumerate}
\item The \textit{unbounding number} $\mathfrak{b}_\kappa$ is the least cardinality of a family $A\subseteq \soafft{\kappa}{\kappa}$ so that for any $f\in \soafft{\kappa}{\kappa}$ there is a $g\in A$ so that $\card{\soast{\alpha < \kappa}{f(\alpha) < g(\alpha)}} = \kappa$.
\item The  \textit{dominating number} $\mathfrak{d}_\kappa$ is the least cardinality of a family $A\subseteq \soafft{\kappa}{\kappa}$ so that for any $f\in \soafft{\kappa}{\kappa}$ there is a $g\in A$ so that $\card{\soast{\alpha < \kappa}{g(\alpha) < f(\alpha)}} < \kappa$.
\item The \textit{reaping number}  $\mathfrak{r}_\kappa$ (sometimes called ``refining number'' or ``unsplitting number'') is the least cardinality of a family $A\subseteq [\kappa]^{\kappa}$ such that for any $X\in [\kappa]^\kappa$ there is $Y\in A$ so that $\min(\card{Y\cap X}, \card{Y\setminus X}) < \kappa$.
\item The \textit{splitting number} $\mathfrak{s}_\kappa$ is the least cardinality of a family $A\subseteq [\kappa]^{\kappa}$ such that for any $X\in [\kappa]^{\kappa}$ there is $Y\in A$ so that $\card{X\cap Y} = \card{X\setminus Y} = \kappa$.
%\item $\non(\meagre)$ is the least cardinality of a nonmeager set in ${}^\omega2$.
\item $\cof(\meagre)$ is the least cardinality of a family $\mathcal{F}$ of meag\re\ subsets of $\soafft{\omega}{2}$ so that for any meag\re\ subset $A$ of $\soafft{\omega}{2}$ there is $B\in\mathcal{F}$ with $A\subseteq B$.
\end{enumerate}

\begin{figure}
\begin{tikzpicture}[char/.style={}, inner sep=6dd]
\pgftransformscale{1.5}
\path node (c) at (0,9) [char] {$\mathfrak{c}$}
 node (i) at (1,8) [char] {$\mathfrak{i}$}
 node (rq) at (2,7) [char] {$\mathfrak{r}_{\mathbb{Q}}$}
 node (cm) at (3,6) [char] {$\cof(\meagre)$}
 node (cn) at (3,8) [char] {$\cof(\mzero)$}
 node (d) at (2,5) [char] {$\mathfrak{d}$}
 node (r) at (1,5) [char] {$\mathfrak{r}$}
 node (u) at (1,6) [char] {$\mathfrak{u}$}
 node (b) at (1,4) [char] {$\mathfrak{b}$}
 node (am) at (1,3) [char] {$\add(\meagre)$}
 node (s) at (2,3) [char] {$\mathfrak{s}$}
 node (t) at (1,1) [char] {$\mathfrak{t}$}
 node (an) at (0,1) [char] {$\add(\mzero)$}
 node (sq) at (1,2) [char] {$\mathfrak{s}_\mathbb{Q}$}
 node (st) at (0,3) [char] {$\stick$}
 node (ao) at (1,0) [char] {$\aleph_1$};
\path	(c) edge (i)
          (c) edge (cn)
          (c) edge (u)
          (u) edge (r)
          (cn) edge (cm)
          (i) edge (rq)
          (rq) edge (r)
	(rq) edge (cm)
	(r) edge (b)
	(d) edge (b)
	(d) edge (s)
	(b) edge (am)
	(am) edge (sq)
	(sq) edge (t)
	(rq) edge (cm)
	(cm) edge (d)
	(an) edge (ao)
	(st) edge (an)
	(st) edge (t)
	(am) edge (an)
	(c) edge (st)
	(s) edge (sq)
	(t) edge (ao);
\end{tikzpicture}
\label{vandouwen}
\caption{A diagram}
\end{figure}

For every cardinal characteristic $\mathfrak{cc} \in \{\mathfrak{b}, \mathfrak{d}, \mathfrak{r}, \mathfrak{s}\}$ the meaning of $\mathfrak{cc}$ is $\mathfrak{cc}_\omega$. It can be proven by $\ZFC$ that $\mathfrak{b}_\kappa \leqslant \mathfrak{d}_\kappa$, $\mathfrak{s}_\kappa \leqslant \mathfrak{d}_\kappa$, and $\mathfrak{b}_\kappa \leqslant \mathfrak{r}_\kappa$ %, and $\cof(\meagre)=\max\{\non(\meagre),\mathfrak{d}\}$
hold for indecomposable $\kappa$ (cf. Blass's chapter \cite{010B0} in the Handbook of Set Theory for $\kappa = \omega$ which easily general\iz es). These characteristics each have value at most $2^\kappa$, moreover, if $\kappa$ is both regular and uncountable, then $\mathfrak{s}_\kappa \leqslant \mathfrak{b}_\kappa$, \cf\ \cite{017RS0}.

We also consider another class of cardinal characteristics, sticks, we follow the terminology established in \cite{006B1}.
\begin{align}
\label{generalised stick-definition}
\stick(\kappa, \lambda) \dfeq &\min\soast{\card{X}}{X \subseteq [\lambda]^\kappa \wedge \forall y \in [\lambda]^\lambda \exists x \in X : x \subseteq y} \text{ for both } \kappa \text{ and } \lambda \text{ cardinals with } \kappa \leqslant \lambda.
\end{align}
We will write $\stick(\kappa)$ for $\stick(\kappa, \kappa^+)$. $\stick$ simply stands for $\stick(\aleph_0)$, \cf\ \cite{997FSS0} and sometimes for $\stick(\aleph_0) = \aleph_1$, \cf\ \cite{978BGKT0}. Note that %$\stick(\kappa, \lambda)$ is also known as $\mathcal{D}(\lambda, \kappa)$, \cf\ \cite{016K0} and that
$\stick(\kappa, \lambda) = \mu$ is denoted by $A(\lambda, \mu, \kappa, \kappa^+)$ in \cite{976B2}.

$\stick(\kappa)$ is not typically considered a cardinal characteristic of $2^\kappa$ since it involves the combinatorics at $\kappa^+$, but as such characteristics it also takes a value greater than $\kappa$ and at most $2^\kappa$. Truss \cite{983T1} showed that if $\stick = \aleph_1$, then either the covering number of the meagre ideal or the covering number of the Lebesgue-null ideal is also $\aleph_1$. Brendle \cite{006B1} further considered the relationship between $\stick$ and other cardinal characteristics.

Now we give an overview of our main results and provide some context.

By the Dushnik--Miller Theorem---\cf\ \cite[Theorem 5.23]{941DM0}---we have $\omega_1 \longrightarrow (\omega_1, \omega+1)^2$. From $\CH$, Hajnal \cite{960H0} proved that $\omega_1 \doesntarrow (\omega_1, \omega+2)^2$. %Todor\v{c}evi\'c \cite{989T0} reduced the hypothesis and showed if $\mathfrak{b}=\aleph_1$, then $\omega_1 \doesntarrow (\omega_1, \omega+2)^2$ (in fact, $\omega_1 \doesntarrow (\omega_1, (\omega:2))^2$).
A recent paper of Raghavan and Todor\v{c}evi\'c \cite{016RS0} shows that if there is a Suslin tree, then
$\omega_1 \doesntarrow (\omega_1, \omega+2)^2$. We can arrive at the same conclusion from the hypothesis $\stick=\aleph_1$ (Theorem \ref{stickdushnikmiller}).

Starting from $\GCH$, Erd\H{o}s and Hajnal \cite{971EH0} proved $\kappa^+\kappa \doesntarrow (\kappa^+\kappa, 3)^2$ for all cardinals $\kappa$ and Hajnal \cite{971H0} proved $\omega_1^2 \doesntarrow (\omega_1^2, 3)^2$ from $\CH$. The hypothesis of $\CH$ was reduced in several different ways.
\begin{fact}
\label{fact : previous}
\begin{enumerate}
\item (Takahashi \cite{987T1}) \label{item : takahashi} If $\stick=\aleph_1$, then $\omega_1^2 \doesntarrow (\omega_1^2, 3)^2$.
\item (Todor\v{c}evi\'c \cite{989T0}) \label{item : todorcevic} If $\mathfrak{b} = \aleph_1$, then $\omega_1 \doesntarrow (\omega_1, \omega + 2)^2$.
\item (Larson \cite{998L0}) \label{item : larson} If $\kappa$ is regular and $\mathfrak{d}_\kappa = \kappa^+$, then $\kappa^+\kappa \doesntarrow (\kappa^+\kappa, 3)^2$ and $(\kappa^+)^2 \doesntarrow \left((\kappa^+)^2, 3\right)^2$.
\item (Lambie-Hanson and Weinert \cite{017LW0}) \label{item : lh&weinert} If $\mathfrak{b}_\kappa = \stick(\kappa) = \kappa^+$, then $\kappa^+\kappa \doesntarrow (\kappa^+\kappa, 3)^2$.
\end{enumerate}
\end{fact}

Baumgartner and Hajnal \cite{987BH0} proved that $2^\kappa = \kappa^+ = \lambda$ for a regular $\kappa$ implies that $\lambda^2 \doesntarrow (\lambda\kappa, 4)^2$. In this paper, we reduce the hypothesis used to $\mathfrak{d}_\kappa = \kappa^+ = \lambda$ for a regular $\kappa$ (Theorem \ref{theorem : quadruples and scales}), answering a question of Larson, and to $\mathfrak{b}_\kappa = \stick(\kappa)= \kappa^+ = \lambda$ for a regular $\kappa$ (Theorem \ref{theorem : quadruples, the unbounding number and stick}). For $\kappa = \omega$ it is known that for this result as well as for Fact \ref{fact : previous} \eqref{item : larson} \& \eqref{item : lh&weinert} a hypothesis is needed as Baumgartner proved in \cite{989B0} that $\ZFC + \MA_{\aleph_1}$ implies that $\omega_1\omega \longrightarrow (\omega_1\omega, n)^2$ for all natural numbers $n$. For \ref{fact : previous} \eqref{item : todorcevic} this was shown by Todor\v{c}evi\'c in \cite{983T0} and for Fact \ref{fact : previous} \eqref{item : takahashi} this is as of now unknown.

%We also prove the general analogue of \eqref{}, that $\mathfrak{b}_\kappa = \stick(\kappa) = \kappa^+ = \lambda$ implies $\lambda\kappa \doesntarrow (\lambda\kappa, 3)^2$.

For polar\iz ed partition relations, Garti and Shelah \cite{014GS0} proved the following:
\begin{fact}\label{fact:polarized}
\begin{enumerate}
\item If $\kappa<\mathfrak{s}$ and $\cfi(\kappa)>\omega$, then $\binom{\kappa}{\omega} \longrightarrow \binom{\kappa}{\omega}^{1,1}_2$.
\item If $\mathfrak{r}<\kappa\le \mathfrak{c}$ and $\cfi(\kappa)>\mathfrak{r}$, then $\binom{\kappa}{\omega} \longrightarrow \binom{\kappa}{\omega}^{1,1}_2$.
\item If $\mathfrak{r}=\mathfrak{c}$ and $\kappa\in [\cfi(\mathfrak{c}),\mathfrak{c}]$, then $\binom{\kappa}{\omega} \doesntarrow \binom{\kappa}{\omega}^{1,1}_2$.
\end{enumerate}
\end{fact}

In Theorem \ref{negpolarized} we are able to prove that $\mathfrak{d}=\aleph_1$ implies $\binom{\omega_1}{\omega} \doesntarrow \binom{\omega_1}{\omega}^{1,1}_2$ (and in fact something a bit stronger, increasing the number of col\ou rs and obtaining a negative square-bracket relation). As corollaries, we answer several questions from \cite{016GS0}.

In the last section, we relate Luzin sets to partition relations, showing that the existence of a Luzin set suffices for an example of Shelah of a function witnessing $\omega_1\doesntarrow[\omega_1]^2_\omega$ with no triangle having three different col\ou rs.

\section{Partition Relations and the Dominating Number}

We recall the following definition.

\begin{definition}
Let $\alpha$ and $\beta$ both be ordinals. A sequence of functions $\seq{f_\xi}{\xi < \rho}$ in $\soafft{\alpha}{\beta}$ is called a \emph{scale} if the order of eventual domination (\ie\ $f_\gamma < f_\delta$ if there is a $\zeta < \alpha$ such that for all $\xi \in \rho \setminus \zeta$ we have $f_\gamma(\xi) < f_\delta(\xi)$) is a well-order. It is called \emph{unbounded} if for all $f \in \soafft{\alpha}{\beta}$ there is a $\xi < \rho$ such that $f < f_\xi$.
\end{definition}

We will be mostly interested in the case where $\alpha$ is a cardinal and $\alpha = \beta$. Note that what we call a scale is called a \emph{strict scale} in \cite{979C0}. One can inductively define a scale of length $\mathfrak{b}$. If $\mathfrak{b} = \mathfrak{d}$ then one can define an unbounded scale of length $\mathfrak{b}$.

\begin{theorem}
\label{theorem : scale}
Suppose that $\mathfrak{b} = \mathfrak{d}$. Then
\begin{align*}
\binom{\mathfrak{d}}{\omega} \doesntarrow \squarebinom{\mathfrak{b}}{\omega}_{\aleph_0}^{1, 1}
\end{align*}
\end{theorem}

This answers \cite[Problems 3.6 \& 3.10]{016GS0} negatively.

\begin{proof}
Assume that $\mathfrak{b} = \mathfrak{d}$. Then we may construct an unbounded scale $\seq{f_\alpha}{\alpha < \mathfrak{b}}$ in $\soafft{\omega}{\omega}$ and assume \wlg\ that for every $\alpha < \mathfrak{b}$ the sequence $f_\alpha$ is properly increasing and $f_\alpha(0) \in \omega \setminus 1$. Furthermore associate to each countable ordinal $\alpha$ a function $g_\alpha : \omega \longrightarrow \omega$ defined inductively by $g_\alpha(0) \dfeq 0$ and $g_\alpha(n + 1) \dfeq f_\alpha\big(g_\alpha(n)\big)$. Note that for each countable ordinal $\alpha$, the function $g_\alpha$ is properly increasing. Fix a mapping $s : \omega \longrightarrow \omega$ such that for all natural numbers $n$ the set $\pwim{s^{-1}}{n}$ is infinite. Now we define a col\ou ring $\chi$ as follows:
\begin{align}
\chi : \mathfrak{b} \times \omega & \longrightarrow \omega,\\
\opair{\alpha}{k} & \longmapsto s\Big(\min\big(\soast{n < \omega}{k \leqslant g_\alpha(n)}\big)\Big) \notag
\end{align}
Suppose towards a contradiction that $X \in [\mathfrak{b}]^{\mathfrak{b}}$ and $Y \in [\omega]^\omega$ are such that $\pwim{\chi}{X \times Y} \ne \omega$. Consider the function
\begin{align}
h : \omega & \longrightarrow \omega,\\
n & \longmapsto \min(Y \setminus (n+1)).\notag
\end{align}
Note that for all natural numbers $n$ we have $h(n) > n$. Let $i \in \omega \setminus \pwim{\chi}{X \times Y}$. Choose an $\alpha \in X$ such that $f_\alpha$ properly dominates $h$ (here we need that $\seq{f_\alpha}{\alpha < \mathfrak{b}}$ is a scale). Let $j$ be a sufficiently large natural number such that for all $n \in \omega \setminus j$ we have $f_\alpha(n) > h(n)$. Choose a $k \in \pwim{s^{-1}}{i} \setminus (j + 1)$ and define $\ell \dfeq \max\big(Y \cap g_\alpha(k)\big)$. But then we have
\begin{align}
s\Big(\min\big(\soast{n < \omega}{\ell \leqslant g_\alpha(n)}\big)\Big) = \chi(\alpha, \ell) \ne i.
\end{align}
Now this implies that there is an $m < k$ with $\ell \leqslant g_\alpha(m)$ which in turn gives rise to the following:
\begin{align}
Y \setminus (\ell + 1) \ni h(\ell) \leqslant h\big(g_\alpha(m)\big) \leqslant h\big(g_\alpha(k - 1)\big) < f_\alpha\big(g_\alpha(k - 1)\big) = g_\alpha(k).
%g_\alpha(k) = f_\alpha\big(g_\alpha(k - 1)\big) > h\big(g_\alpha(k - 1)\big) \geqslant g\big(g_\alpha(m)\big) \geqslant h(\ell) \in Y \setminus (\ell + 1).
\end{align}
This contradicts the definition of $\ell$ thus proving the Theorem.
\end{proof}

Theorem \ref{theorem : scale} solves \cite[Problem 3.19]{016GS0}. The problem asks whether for a cardinal $\aleph_1<\kappa<\mathfrak{c}$ it is possible for $\binom{\kappa}{\omega} \longrightarrow \binom{\kappa}{\omega}^{1,1}_2$  to be destroyed by the L\'evy collapse of $\mathfrak{c}$ to $\kappa$. We give an affirmative answer by showing that it is consistent that $\aleph_1 < \kappa < \mathfrak{c}$ and $\binom{\kappa}{\omega} \longrightarrow \binom{\kappa}{\omega}^{1,1}_2$ but after L\'evy-collapsing $\mathfrak{c}$ to $\kappa$ one has $\binom{\kappa}{\omega} \doesntarrow \binom{\kappa}{\omega}^{1,1}_2$.

Fischer and Stepr\={a}ns proved in \cite{008FS0}, that $\aleph_1 < \mathfrak{b} < \mathfrak{s}$ is consistent. L\'evy-collapsing $\mathfrak{c}$ to $\mathfrak{b}$ over this model yields a model with a scale of length $\mathfrak{b}$ which by Theorem \ref{theorem : scale} implies $\binom{\mathfrak{b}}{\omega} \doesntarrow \binom{\mathfrak{b}}{\omega}$.

Note that this has been solved independently by Garti and Shelah in \cite{018GS0}.

An obvious question is whether or not the hypothesis in Theorem \ref{theorem : scale} is necessary. The answer to this question is negative.

\begin{theorem}
\label{nnegcohen} It is consistent that $\mathfrak{b}<\mathfrak{d}$ and $\left( \begin{smallmatrix} \mathfrak{d} \\ \omega \end{smallmatrix} \right) \nrightarrow \left[ \begin{smallmatrix} \omega_1 \\ \omega \end{smallmatrix} \right]^{1,1}_{\aleph_0}$.
This amounts also to $\left( \begin{smallmatrix} \mathfrak{d} \\ \omega \end{smallmatrix} \right) \nrightarrow \left[ \begin{smallmatrix} \mathfrak{b} \\ \omega \end{smallmatrix} \right]^{1,1}_{\aleph_0}$.
The distance between $\mathfrak{b}$ and $\mathfrak{d}$ can be arbitrarily large.
\end{theorem}

\begin{proof}
We assume the continuum hypothesis in the ground model, and we choose $\lambda>\aleph_1$ so that $\lambda^\omega=\lambda$. Let $\mathbb{Q}$ be the following forcing notion. We say that $p\in\mathbb{Q}$ iff $p$ is a partial function from $\omega$ into $\omega$ and $|\rm dom(p)|<\aleph_0$. If $p,q\in\mathbb{Q}$ then $p\leq q$ iff $p\subseteq q$.

Let $\mathbb{P}$ be the finite support iteration $\langle\mathbb{P}_\alpha, \name{\mathbb{Q}}_\beta: \alpha\leq\lambda, \beta<\lambda\rangle$, when $\name{\mathbb{Q}}_\beta$ is a $\mathbb{P}_\beta$-name of the forcing $\mathbb{Q}$ at every successor stage. It is known that $\mathbb{Q}$ is isomorphic to the usual Cohen forcing (see, e.g., Halbeisen, Combinatorial set theory) and hence $\mathbb{P}$ is isomorphic to the forcing which adds $\lambda$ many Cohen reals. In particular, $\mathbb{P}$ is $ccc$ so no cardinal is collapsed and no cofinality is changed.

Let $G\subseteq\mathbb{P}$ be generic. It follows that $V[G]\models \mathfrak{b}=\aleph_1<\lambda=\mathfrak{d}=\mathfrak{c}$. We shall prove that $V[G]\models \left( \begin{smallmatrix} \lambda \\ \omega \end{smallmatrix} \right) \nrightarrow \left[ \begin{smallmatrix} \omega_1 \\ \omega \end{smallmatrix} \right]^{1,1}_{\aleph_0}$. For this end, let $(f_\alpha:\alpha<\lambda)$ be an enumeration of the Cohen reals we added, so $f_\alpha\in{}^\omega \omega$ for every $\alpha<\lambda$.
We define a name $\name{c}$ of a coloring from $\lambda\times\omega$ into $\aleph_0$ as follows. Given $\alpha<\lambda$ and $n\in\omega$ we let $\name{c}(\alpha,n)=\name{f_\alpha}(n)$. We claim that $\name{c}$ exemplifies the negative relation to be proved.

For this, assume that $\name{A}\in[\lambda]^{\aleph_1}$ and $\name{B}\in[\omega]^\omega$. Assume towards contradiction that there exists a condition $p_0\in\mathbb{P}$ which forces that $\name{c}``(\name{A}\times\name{B})\neq \aleph_0$. We may extend $p_0$ into a condition $p$ such that $p\Vdash\check{m}\notin \name{c}``(\name{A}\times\name{B})$ for some $m\in\omega$. The idea of the proof is to find some $\alpha\in\lambda$ and $n\in\omega$ for which no value has been fixed yet, to force them into $\name{A}\times\name{B}$ and then to assign the value $m$ to $\name{c}(\alpha,n)$.

Firstly, for every $n\in\omega$ let $\mathcal{A}_n$ be a maximal anti-chain which decides the statement $\check{n}\in\name{B}$. The size of each $\mathcal{A}_n$ is at most $\aleph_0$, since $\mathbb{P}$ is $ccc$. Likewise, $|{\rm dom}(q)|<\aleph_0$ for every $q\in\mathcal{A}_n$, so the set $U=\bigcup\{{\rm dom}(q):q\in\mathcal{A}_n,n\in\omega\}$ is countable. Recall that $\name{A}\in[\lambda]^{\aleph_1}$, so $\Vdash_{\mathbb{P}} \name{A}\nsubseteq U$. Fix any ordinal $\alpha\in\lambda$ so that $p\nVdash \check{\alpha}\notin\name{A}$. Now we choose a condition $q\geq p$ such that $q\Vdash\check{\alpha}\in\name{A}$, and we may assume without loss of generality that $\alpha\in{\rm dom}(q)$.

Secondly, we need the apropriate $n\in\omega$. Choose $n_0\in\omega$ such that $\sup({\rm dom}(q(\alpha)))<n_0$. We choose some $n\in\omega$ and a condition $r\geq q\upharpoonright U$ such that ${\rm dom}(r)\subseteq U$ and $r\Vdash n_0<n\in\name{B}$. Let $s$ be $r\cup q\upharpoonright({\rm dom}(q) \setminus U)$ and let $t$ be $s\cup\langle\alpha,n,m\rangle$. It follows that $t\Vdash\name{c}(\alpha,n)=m$, a contradiction.
\end{proof}

Now we use Theorem \ref{theorem : scale}  to settle another question from \cite{018GS0}. Recall from Fact \ref{fact:polarized} that if $\mathfrak{r}<\kappa\leqslant\mathfrak{c}$ and $\cfi(\kappa)>\mathfrak{r}$, then $\binom{\kappa}{\omega} \longrightarrow \binom{\kappa}{\omega}$. Asking what happens when $\aleph_0<\cfi(\kappa)\leqslant\mathfrak{r}$ gives rise to the following problem, \cf\ \cite[Question 1.8(a)]{018GS0}:

\begin{question}
\label{q1} Is it consistent that $2^{\aleph_0}=\lambda>\mathfrak{r}, \cfi(\lambda)>\aleph_0$ and $\binom{\lambda}{\omega}\doesntarrow \binom{\lambda}{\omega}^{1,1}_2$?
\end{question}

We shall give a negative answer, which is a bit surprising.
It demonstrates the fact that the duality between $\mathfrak{r}$ and $\mathfrak{s}$ is not totally complete.
Recall that $\mathfrak{i}$, the independence number, is the minimal size of a maximal independent family in $[\omega]^\omega$.
It is well known that $\mathfrak{r},\mathfrak{d}\leqslant\mathfrak{i}$, see e.g. \cite{010B0}. It is also known that one can force $\mathfrak{i}=\aleph_1$ and $\lambda=2^{\aleph_0}$ provided that $\cfi(\lambda)>\omega$.

\begin{theorem}
\label{a1}
  One can force $\mathfrak{r}<\lambda, \cfi(\lambda)>\omega$ and yet $\binom{\lambda}{\omega}\doesntarrow\binom{\lambda}{\omega}^{1,1}_2$.
\end{theorem}

\begin{proof}
Let us prove the following general assertion.
\begin{claim*}
If $\displaystyle\binom{\kappa}{\omega}\doesntarrow\binom{\kappa}{\omega}^{1,1}_2$ and $\lambda>\cfi(\lambda)=\kappa$ then $\displaystyle\binom{\lambda}{\omega}\doesntarrow\binom{\lambda}{\omega}^{1,1}_2$.
\end{claim*}

\begin{proofofclaim}
For this, choose a continuous increasing sequence of cardinals $\langle\vartheta_\varepsilon:\varepsilon<\kappa\rangle$ such that $\lambda=\bigcup_{\varepsilon<\kappa}\vartheta_\varepsilon$ and $\vartheta_{\varepsilon+1}$ is a regular cardinal for every $\varepsilon<\kappa$.
Define the intervals mapping $h:\lambda\rightarrow\kappa$ by $h(\alpha)=\min\{\varepsilon<\kappa:\vartheta_\varepsilon\leqslant\alpha<\vartheta_{\varepsilon+1}\}$.

Choose a col\ou ring $c:\kappa\times\omega\rightarrow 2$ which exemplifies the negative relation $\binom{\kappa}{\omega}\doesntarrow\binom{\kappa}{\omega}^{1,1}_2$. For every $\alpha<\lambda, n\in\omega$ let $d(\alpha,n)$ be $c(h(\alpha),n)$, so $d:\lambda\times\omega\rightarrow 2$. We claim that $d$ exemplifies the negative relation $\binom{\lambda}{\omega}\doesntarrow\binom{\lambda}{\omega}^{1,1}_2$.

Assume toward contradiction that $A\in[\lambda]^\lambda, B\in[\omega]^\omega$ and $d\upharpoonright(A\times B)$ is constant.
Let $H = \{h(\alpha):\alpha\in A\}$.
Since $A$ is unbounded in $\lambda$ and $\cfi(\lambda)=\kappa$ we see that $H\in[\kappa]^\kappa$.
We shall prove that $c\upharpoonright(H\times B)$ is constant, thus arriving at a contradiction.

Let $j$ be the constant value of $d$ over $A\times B$.
Choose any $\varepsilon\in H, n\in\omega$.
Let $\alpha$ be an element of $A$ such that $\varepsilon=h(\alpha)$.
From the definition of $c$ it follows that $c(\varepsilon,n)=c(h(\alpha),n)=d(\alpha,n)=j$, so the claim is proved.
\end{proofofclaim}

Now force $\mathfrak{i}=\aleph_1$ while $\lambda=2^{\aleph_0}=\aleph_{\omega_1}$.
As noted above, $\mathfrak{d}\leqslant\mathfrak{i}$ hence $\mathfrak{b}=\mathfrak{d}=\omega_1$ so $\binom{\omega_1}{\omega}\doesntarrow\binom{\omega_1}{\omega}^{1,1}_2$.
By the above statement, $\binom{\lambda}{\omega}\doesntarrow\binom{\lambda}{\omega}^{1,1}_2$ as well, so the proof is accomplished.
\end{proof}

Theorem \ref{theorem : scale} can be improved if $\mathfrak{d}=\aleph_1$.

\begin{theorem}
\label{negpolarized}
%Suppose that $\linr = \aleph_1$. Then
Suppose that $\mathfrak{d} = \aleph_1$. Then
\begin{align*}
\binom{\omega_1}{\omega_1} \doesntarrow \sqtarg{\omega_1}{\omega}{\omega}{\omega_1}^{1,1}_{\aleph_0}.
\end{align*}
\end{theorem}

\begin{proof}
We use the following lemma without proof which comes from the construction of an Aronszajn tree.
\begin{lemma}\label{aronszajn}
There exists a sequence of functions $\langle g_\alpha:\alpha<\omega_1\rangle$ so that for all countable ordinals $\alpha$:
\begin{enumerate}
\item $g_\alpha:\alpha\rightarrow\omega$ is injective.
\item \label{condition : 2nd Aronszajn} For any countable ordinal $\delta$, the set $\soast{g_\alpha\upharpoonright\delta}{\alpha<\omega_1}$ is countable.
\end{enumerate}
\end{lemma}
%$\linr = \aleph_1$ implies $\mathfrak{d} = \aleph_1$ by \cite[Theorem 2.2]{016MST0}.
Fix a sequence of functions $\langle g_\alpha:\alpha<\omega_1\rangle$ as in Lemma \ref{aronszajn}.

Let $s : \omega \longrightarrow \omega$ be such that for all natural numbers $i$ the set $\pwim{s^{-1}}{\{i\}}$ is infinite. Moreover, let $\seq{f_\alpha}{\alpha < \omega_1}$ be a scale. We may assume \wlg\ that $f_\alpha(n) > n$ for all countable ordinals $\alpha$ and all natural numbers $n$. Define a sequence $\seq{k_\alpha}{\alpha < \omega_1}$ by setting ($k_{\omega\alpha + i + 1} \dfeq f_\alpha(k_{\omega\alpha + i})$ for all natural numbers $i$ and $k_{\omega\alpha} \dfeq 0$) for all countable ordinals $\alpha$. Note that for every countable ordinal $\alpha$ the sequence $\seq{k_{\omega\alpha + i}}{i < \omega}$ is strictly increasing. Now we define a col\ou ring as follows:
\begin{align}
\chi : \omega_1 \times \omega_1 & \longrightarrow \omega, \\
\opair{\alpha}{\beta} & \longmapsto \notag
\begin{cases}
 s\big(\min\soast{i < \omega}{g_\alpha(\beta) < k_{\omega\alpha + i}}\big) & \text{if } \beta < \alpha, \\
 \chi(\beta, \alpha) & \text{if } \alpha < \beta, \\
 0 & \text{else.}
\end{cases}
\end{align}
Now arbitrarily choose sets $X \in [\omega_1]^{\omega_1}$ and $Y \in [\omega_1]^\omega$ as well as a col\ou r $n$. Let $\gamma \dfeq \sup Y$. By condition \eqref{condition : 2nd Aronszajn} of the statement of Lemma \ref{aronszajn}, the set $H\dfeq\soast{g_\delta \upharpoonright \gamma}{\delta < \omega_1}$ is countable. Let $\seq{h_i}{i < \omega}$ be an enumeration of $H$. We define a function
\begin{align}
h : \omega & \longrightarrow \omega \\
k & \longmapsto \max_{i < k}\big(\min\soast{j \in \omega \setminus k}{h_i^{-1}(j) \in Y}\big) \notag
\end{align}
Let $\alpha \in X \setminus (\gamma + 1)$ be sufficiently large such that $f_\alpha$ eventually properly dominates $h$. Let $\ell$ be such that $h_\ell = g_\alpha \upharpoonright \gamma$. Let $\bar{\ell} \in \omega \setminus (\ell + 2)$ be such that $f_\alpha(i) > h(i)$ for all $i \in \omega \setminus \bar{\ell}$ and choose an $m \in \omega \setminus \bar{\ell}$ such that $s(m) = n$. We have
\begin{align}
k_{\omega\alpha + m} = f_\alpha(k_{\omega\alpha + m - 1}) > h(k_{\omega\alpha + m - 1}) \geqslant \min\soast{j \in \omega \setminus k_{\omega\alpha + m - 1}}{h_\ell^{-1}(j) \in Y}.
\end{align}
In other words, there is an ordinal $\beta \in Y$ such that $h_\ell(\beta) = g_\alpha(\beta) \in k_{\omega\alpha + m} \setminus k_{\omega\alpha + m - 1}$. But then, clearly, $\chi(\alpha, \beta) = s(m) = n$.
\end{proof}

With a creature forcing construction, Shelah proved that it is consistent that $\aleph_1=\mathfrak{b}<\mathfrak{s}$. Therefore, the hypothesis of Theorem \ref{negpolarized} cannot be weakened to $\mathfrak{b}=\aleph_1$, i.e., it is consistent that $\mathfrak{b}=\aleph_1$ and $\binom{\omega_1}{\omega}\rightarrow \binom{\omega_1}{\omega}^{1, 1}_2$.

We answer Jean Larson's question from \cite[page 112]{998L0} affirmatively.
\begin{theorem}
\label{theorem : quadruples and scales}
Suppose that $\kappa$ is regular and $\lambda = \kappa^+ = \mathfrak{d}_\kappa$. Then $\lambda^2 \doesntarrow (\lambda\kappa, 4)^2$.
\end{theorem}

%Recall that for an ordinal $\kappa$ and a regular $\lambda$ we write $E^\kappa_\lambda$ for $\soast{\xi < \kappa}{\cfi(\xi) = \lambda}$.

\begin{proof}

For every ordinal $\xi < \lambda$ we fix the following things:
\begin{align}
\text{An injection } b_\xi & : \xi \hookrightarrow \kappa, \\
\text{an increasing function } d_\xi & : \kappa \longrightarrow \kappa \text{, and} \\
\text{sequences of ordinals } e(\xi, \cdot), f(\xi, \cdot) & : \kappa \longrightarrow \xi.
\end{align}
More specifically we do this in a way such that
\begin{align}
\soast{d_\xi}{\xi < \lambda} & \text{ is a scale of length } \lambda \text{ in } \soafft{\kappa}{\kappa} \text{ such that} \\
\label{condition : increasing}
& \forall \xi < \lambda\forall\nu < \kappa(d_\xi(\nu) > \nu) \text{ and} \\
g_\xi : \kappa & \longrightarrow \xi \times \xi, \\
\nu & \longmapsto \opair{e(\xi, \nu)}{f(\xi, \nu)} \\
\text{satisfies } & \forall x \in \xi \times \xi : \card{\pwim{g_\xi^{-1}}{x}} = \kappa.
\end{align}

We are going to define a graph $\opair{\lambda^2}{\Delta}$ where $\{\lambda \alpha + \beta, \lambda \gamma + \delta\} \in \Delta$ together with $\lambda \alpha + \beta < \lambda \gamma + \delta$ implies $\alpha < \gamma < \delta < \beta$.
(Note that every graph corresponds to a $2$-col\ou ring of its vertex-set where a pair of vertices gets one col\ou r if both vertices are connected to each other by an edge and the other if they are not.)
Given a $\beta < \lambda$ and an $\alpha < \beta$ we inductively define the sets $\Gamma(\lambda \alpha + \beta) = \soast{\lambda \gamma + \delta}{\alpha \leqslant \gamma \leqslant \delta \leqslant \beta \wedge \{\lambda\alpha + \beta, \lambda\gamma + \delta\} \in \Delta}$. The induction is layered. The top layer of the induction has length $\lambda$ and in step $\beta < \lambda$ we define all sets $\Gamma(\lambda \alpha + \beta)$ with $\alpha < \beta$. The second layer of the induction has length $\beta$ and in step $\alpha < \beta$ we define the set $\Gamma(\lambda \alpha + \beta)$. The third layer of the induction has length $\kappa$ where in every step fewer than $\kappa$ ordinals are added to $\Gamma(\lambda \alpha + \beta)$ as elements.

Suppose that we are in step $\zeta$ of the third layer of the induction and previous steps $\mu < \zeta$ have added the ordinals $\lambda \gamma_{\mu, \nu} + \delta_{\mu, \nu, \xi}$ $(\nu < \vartheta_\mu, \xi < \iota_{\mu, \nu})$ to $\Gamma(\lambda \alpha + \beta)$ and have defined $\sigma_\mu$. As induction hypothesis we assume that for all $\mu < \zeta$ we have ($\vartheta_\mu < \kappa$ and $\forall \nu < \vartheta_\mu : \iota_{\mu, \nu} < \kappa$). Using $\kappa$'s regularity we now choose a $\rho_\zeta < \kappa$ satisfying
\begin{align}
\label{condition : ascending}
\rho_\zeta & > \sup_{\xi < \zeta} \sigma_\xi, \\
\label{condition : greater than gammas}
b^{-1}_{e(\beta, \zeta)}(\rho_\zeta) & > \sup(e(\beta, \zeta) \cap \soast{\gamma_{\mu, \nu}}{\mu < \zeta \wedge \nu < \vartheta_\mu}), \\
\label{condition : greater than alphas}
b^{-1}_{e(\beta, \zeta)}(\rho_\zeta) & > \sup(\soast{\pwim{b_{\gamma_{\mu, \nu}}^{-1}}{b_{\gamma_{\mu, \nu}}(\alpha)}}{\mu < \zeta \wedge \nu < \vartheta_\mu \wedge e(\beta, \zeta) \leqslant \gamma_{\mu, \nu}}) , \\
\label{condition : greater than gamma-indices}
b^{-1}_{e(\beta, \zeta)}(\rho_\zeta) & > \sup(\soast{\pwim{b_{\gamma_{\mu, \nu}}^{-1}}{b_{\gamma_{\mu, \nu}}(\gamma_{\mu', \nu'})}}{\mu , \mu' < \zeta \wedge \nu, \nu' < \vartheta_\mu \wedge e(\beta, \zeta) \leqslant \gamma_{\mu', \nu'} < \gamma_{\mu, \nu}}).
\end{align}
and set $\sigma_\zeta \dfeq d_\beta(\rho_\zeta)$. Note that together with \eqref{condition : increasing} and \eqref{condition : ascending} this implies that
\begin{align}
\forall \xi < \kappa(\xi < \rho_\xi < \sigma_\xi).
\end{align}

Now let $\seq{\gamma_{\zeta, \nu}}{\nu < \vartheta_\zeta}$ be the increasing enumeration of the set
\begin{align}
\label{assignment : gammas}
\soast{b^{-1}_{e(\beta, \zeta)}(\xi)}{\xi \in \sigma_\zeta \setminus \rho_\zeta \wedge b^{-1}_{e(\beta, \zeta)}(\rho_\zeta) \leqslant b^{-1}_{e(\beta, \zeta)}(\xi)}
\end{align}
Now inductively for every $\mu < \vartheta_\zeta$ let
\begin{align}
\label{assignment : going right}
\varphi_{\zeta, \mu} \dfeq \sup\soast{\max(\sigma_\zeta, b_{f(\beta, \zeta)}(\delta))}{\exists \nu, \xi((\nu < \zeta \wedge \xi < \vartheta_\nu) \vee (\nu = \zeta \wedge \xi < \mu)) \wedge \exists \tau < \iota_{\nu, \xi} : \lambda \gamma_{\zeta, \mu} + \delta \in \Gamma(\lambda \gamma_{\nu, \xi} + \delta_{\nu, \xi, \tau})},
\end{align}
let $\iota_{\zeta, \mu}$ be the least ordinal $\iota$ such that $\varphi_{\zeta, \mu} + \iota \geqslant d_\beta(\varphi_{\zeta, \mu})$ and $\delta_{\zeta, \mu, \nu} \dfeq b^{-1}_{f(\beta, \zeta)}(\varphi_{\zeta, \mu} + \nu)$ for every $\nu < \iota_{\zeta, \mu}$.

This finishes the definition of $\Gamma(\lambda \alpha + \beta)$ and thereby the definition of $\Delta$. Now we check that $\Delta$ witnesses $\lambda^2 \doesntarrow (\lambda\kappa, 4)^2$.

First assume that there was a $Q \in [\lambda^2]^4$ such that $[Q]^2 \subseteq \Delta$. Suppose that $Q = \{\lambda \alpha_0 + \beta_0, \lambda \alpha_1 + \beta_1,  \lambda \alpha_2 + \beta_2, \lambda \alpha_3 + \beta_3\}$ with $\alpha_0 < \alpha_1 < \alpha_2 < \alpha_3$. Then, \wlg, we get $\alpha_0 < \alpha_1 < \alpha_2 < \alpha_3 < \beta_3 < \beta_2 < \beta_1 < \beta_0$. Now first assume that there are $\ell \in 4 \setminus 2$ and $k \in \ell \setminus 1$ such that $\{\lambda \alpha_0 + \beta_0, \lambda \alpha_k + \beta_k\}$ was added to $\Delta$ before $\{\lambda \alpha_0 + \beta_0, \lambda \alpha_\ell + \beta_\ell\}$---that is, $\lambda \alpha_k + \beta_k$ was added to $\Gamma(\lambda \alpha_0 + \beta_0)$ before $\lambda \alpha_\ell + \beta_\ell$. So assume that $\lambda\alpha_k + \beta_k$ was added in induction step $\zeta$ and $\lambda\alpha_\ell + \beta_\ell$ was added in induction step $\mu$ where $\zeta \leqslant \mu < \kappa$. Let $\nu < \vartheta_\zeta$ be such that $\gamma_{\zeta, \nu} = \alpha_k$ and let $\xi < \iota_{\zeta, \nu}$ be such that $\delta_{\zeta, \nu, \xi} = \beta_k$. Similarly, let $\tau < \vartheta_\mu$ be such that $\gamma_{\mu, \tau} = \alpha_\ell$ and let $\psi < \iota_{\mu, \tau}$ be such that $\delta_{\mu, \tau, \psi} = \beta_\ell$. Then $b_{f(\beta, \mu)}(\beta_\ell) = \varphi_{\mu, \tau} + \psi$. As
\begin{align}
\lambda \gamma_{\mu, \tau} + \delta_{\mu, \tau, \psi} & = \lambda \alpha_\ell + \beta_\ell \in \Gamma(\lambda \alpha_k + \beta_k) = \Gamma(\lambda \gamma_{\zeta, \nu} + \delta_{\zeta, \nu, \xi}) \\
\text{and either } & (\zeta < \mu \wedge \nu < \vartheta_\zeta) \text{ or } (\zeta = \mu \wedge \nu < \tau) \text{, by \eqref{assignment : going right}
 we have} \\
b_{f(\beta, \mu)}(\beta_\ell) = b_{f(\beta, \mu)}(\delta_{\mu, \tau, \psi}) & = \varphi_{\mu, \tau} + \psi \geqslant \varphi_{\mu, \tau} > b_{f(\beta, \mu)}(\delta_{\mu, \tau, \psi}) = b_{f(\beta, \mu)}(\beta_\ell),
\end{align}
a contradiction.

Now assume that this is not the case, \ie\ for $k \in 4 \setminus 1$ the set $\lambda \alpha_k + \beta_k$ was added to $\Gamma(\lambda \alpha_0 + \beta_0)$ in induction step $\mu_k$ where $\mu_3 \leqslant \mu_2 \leqslant \mu_1 < \kappa$. Also, for $k \in 4 \setminus 1$, let $\nu_k < \vartheta_{\mu_k}$ be such that $\alpha_k = \gamma_{\mu_k, \nu_k}$. Then, for $k \in \{1, 2\}$, if $\mu_k = \mu_{k + 1}$ then $\nu_{k + 1} < \nu_k$. We have
\begin{align}
\label{inequality : mus}
\mu_3 < \mu_2 < \mu_1 < \kappa
\end{align}
because if for some $k \in \{1, 2\}$ we have $\mu_k = \mu_{k + 1}$, then
$\alpha_{k + 1} = \gamma_{\mu_{k + 1}, \nu_{k +  1}} = \gamma_{\mu_k, \nu_{k + 1}} < \gamma_{\mu_k, \nu_k} = \alpha_k$, a contradiction. This proves \eqref{inequality : mus}.

Furthermore we have
\begin{align}
\label{result : sufficiently high}
e(\beta_\ell, \mu_k) \leqslant \alpha_{k + 1} \text{ for both } k \in \{1, 2\} \text{ and } \ell < k
\end{align}
since if there was a $k \in \{1, 2\}$ with $e(\beta_\ell, \mu_k) > \alpha_{k + 1} = \gamma_{\mu_{k + 1}, \nu_{k + 1}}$, then by \eqref{condition : greater than gammas}, $b^{-1}_{e(\beta_\ell, \mu_k)}(\rho_{\mu_k}) > \gamma_{\mu_{k + 1}, \nu_{k + 1}}$.
By \eqref{assignment : gammas}, we have $\gamma_{\mu_k , \nu_k} \geqslant b^{-1}_{e(\beta_\ell, \mu_k)}(\rho_{\mu_k})$ and hence $\alpha_k = \gamma_{\mu_k, \nu_k} > \gamma_{\mu_{k + 1}, \nu_{k + 1}} = \alpha_{k + 1}$, a contradiction. This proves \eqref{result : sufficiently high}.

Now we get
\begin{align}
\label{result : towards a contradiction}
b_{\alpha_3}(\alpha_2) < b_{\alpha_3}(\alpha_1)
\end{align}
as otherwise  $b_{\alpha_3}(\gamma_{\mu_2, \nu_2}) = b_{\alpha_3}(\alpha_2) \geqslant b_{\alpha_3}(\alpha_1) = b_{\alpha_3}(\gamma_{\mu_1, \nu_1})$ and as $b_{\alpha_3}$ is a bijection, $b_{\alpha_3}(\gamma_{\mu_2, \nu_2}) > b_{\alpha_3}(\gamma_{\mu_1, \nu_1})$. Then, by \eqref{condition : greater than gamma-indices} in combination with \eqref{result : sufficiently high} for $\ell = 0$, we have $b^{-1}_{e(\beta_0, \mu_1)}(\rho_{\mu_1}) > \gamma_{\mu_1, \nu_1}$. But by \eqref{assignment : gammas}, we have $\gamma_{\mu_1, \nu_1} \geqslant b^{-1}_{e(\beta_0, \mu_1)}(\rho_{\mu_1})$, a contradiction. This proves \eqref{result : towards a contradiction}.

Now, by \eqref{condition : greater than alphas} in combination with \eqref{result : sufficiently high} for $\ell = 1$ and $k = 2$ and \eqref{result : towards a contradiction}, we have $b^{-1}_{e(\beta_1, \mu_2)}(\rho_{\mu_2}) > \alpha_2$. But then, by \eqref{assignment : gammas}, we get $\alpha_2 \geqslant b^{-1}_{e(\beta_1, \mu_2)}(\rho_{\mu_2})$, a contradiction.

This refutes the existence of a $Q \in [\lambda^2]^4$ such that $[Q]^2 \subseteq \Delta$.

Now assume that there is an $H \in [\lambda^2]^{\lambda\kappa}$ such that $[H]^2 \subseteq [\lambda^2]^2 \setminus \Delta$. We define
\begin{align}
A \dfeq \soast{\mu < \lambda}{\card{\soast{\nu < \lambda}{\lambda \mu + \nu \in H}} = \lambda}.
\end{align}
It is easy to show that $A \in [\lambda]^\kappa$. Let $\xi \dfeq \sup(A)$ and $\tau \dfeq \min\soast{\mu \in \lambda \setminus \xi}{\forall \nu \in A : \card{\soast{\psi < \mu}{\lambda \nu + \psi \in H}} = \kappa}$. Now we define two functions
\begin{align}
h_0 : \kappa & \longrightarrow \kappa \\
\mu & \longmapsto \min\soast{\nu \in \kappa \setminus \mu}{b_\xi^{-1}(\nu) \in A} \\
h_1 : \kappa & \longrightarrow \kappa \\
\mu & \longmapsto \min\soast{\nu \in \kappa \setminus \mu}{\forall \gamma \in A(b_\xi(\gamma) \leqslant \mu \rightarrow \exists \delta(b_\tau(\delta) \in \nu \setminus \mu \wedge \lambda \gamma + \delta \in H))}.
\end{align}
Let $\alpha \dfeq \min(A)$. Moreover, let $\beta \in \lambda \setminus \max(\xi, \tau)$ be such that $d_\beta$ eventually dominates $h_i$ for both $i < 2$ and $\lambda \cdot \alpha + \beta \in H$. Let $\mu < \kappa$ be such that for both $i < 2$ the sequence $d_\beta$ properly dominates $h_i$ above $\mu$. Let $\zeta \in \kappa \setminus (\max(\mu, b_\xi(\alpha)) + 1)$ such that $g_\beta(\zeta) = \opair{\xi}{\tau}$. We have $\rho_\zeta \geqslant \zeta$. As $d_\beta$ properly dominates $h_0$ at $\rho_\zeta$, the definition of $h_0$ implies $\sigma_\zeta = d_\beta(\rho_\zeta) > \min\soast{\nu \in \kappa \setminus \rho_\zeta}{b_\xi^{-1}(\nu) \in A}$. So there is a $\nu < \vartheta_\zeta$ such that $\gamma_{\zeta, \nu} \in A$, \cf\ \eqref{assignment : gammas}.

Also, $\varphi_{\zeta, \nu} \geqslant \sigma_\zeta \geqslant \rho_\zeta \geqslant \zeta > \mu$. Therefore, $\varphi_{\zeta, \nu} + \iota_{\zeta, \nu} = d_\beta(\varphi_{\zeta, \nu}) > h_1(\varphi_{\zeta, \nu})$. As $b_\xi(\gamma_{\zeta, \nu}) < \sigma_\zeta \leqslant \varphi_{\zeta, \nu}$ there is a $\delta < \tau$ such that $\mu \leqslant b_\tau(\delta) < h_1(\varphi_{\zeta, \nu}) < d_\beta(\varphi_{\zeta, \nu})$ and $\lambda \gamma_{\zeta, \nu} + \delta \in H$. But this means that there is a $\psi < \iota_{\zeta, \nu}$ such that $\delta = \delta_{\zeta, \nu, \psi}$ and $\lambda \gamma_{\zeta, \nu} + \delta_{\zeta, \nu, \psi} \in H \cap \Gamma(\lambda \alpha + \beta) \subseteq \Delta$. As $\lambda \alpha + \beta \in H$ we have $\{\lambda \alpha + \beta, \lambda \gamma_{\zeta, \nu} + \delta_{\zeta, \nu, \psi}\} \in [H]^2$ showing $[H]^2 \not \subseteq [\lambda^2]^2 \setminus \Delta$, a contradiction.
\end{proof}

\begin{corollary}
If $\kappa$ is regular  and $\lambda = \kappa^+ = \mathfrak{d}_\kappa$, then $\alpha \doesntarrow (\lambda\kappa, 4)^2$ for all $\alpha < \lambda^2\kappa$.
\end{corollary}

\section{Partition Relations from the Stick-Principle}\label{sec : stick}
The principle $\stick(\kappa)=\kappa^+$ can be used directly to guess an infinite subset of a homogeneous set. In some cases, a diagonal process which takes care of all of the guesses can be used to obtain a negative partition relation. This is the idea behind Takahashi's result that $\stick=\aleph_1$ implies $\omega_1^2 \doesntarrow (\omega_1^2, 3)^2$ (and the result generalizes straightforwardly to arbitrary $\kappa$).

Combining this method with the argument of Todor\v{c}evi\'c \cite{989T0} gives the following:

\begin{theorem}\label{stickdushnikmiller}
Let $\kappa$ be a regular cardinal. Then $\stick(\kappa) = \kappa^+$ implies $\kappa^+ \doesntarrow (\kappa^+, (\kappa:2))^2$.
\end{theorem}

\begin{proof}
In the proof below, we use the adverb \emph{almost} to mean ``modulo sets of size $<\kappa$''.

Let $\langle e_\alpha:\alpha<\kappa^+\rangle$ be a sequence of order-type $\kappa$ subsets witnessing $\stick(\kappa)=\kappa^+$. Let $\langle d_\alpha\mid \alpha<\kappa^+\rangle$ be an enumeration of $\left\{\bigcup_{\beta\in e_\gamma} e_\beta\mid \gamma<\kappa^+\right\}$
so that $d_\alpha\subseteq \alpha$. Note that this also witnesses $\stick=\kappa^+$. Define by induction on $\alpha<\kappa^+$ a sequence $\langle B_\alpha\mid \alpha<\kappa^+\rangle$ of size $\kappa$ subsets of $\kappa^+$ so that:
\begin{enumerate}
\item $B_\alpha\subseteq \alpha$.
\item $|B_\alpha\cap B_{\alpha'}|<\kappa$ if $\alpha\neq\alpha'$.
\item $B_\alpha\cap d_\beta\neq \emptyset$ if $\beta<\alpha$ and $d_\beta$ is not contained in any $<\kappa$-union of sets of the form $B_{\alpha'}$ for $\alpha'<\alpha$.
\end{enumerate}
For the construction at stage $\alpha$, let $g_\alpha:\kappa\rightarrow\alpha$ be a surjection. Then inductively construct $B_\alpha$ to be the set $\soast{x^\alpha_i}{i<\kappa}$, where $$x^\alpha_i\in d_{g_\alpha(i)}\setminus \bigcup_{j<i} B_{g_\alpha(j)}$$
if such exists.

Let the coloring be defined by $c(\alpha',\alpha)=1$ iff $\alpha'<\alpha$ and $\alpha'\in B_\alpha$. There is no $1$-homogenous set $\soast{ \alpha_\xi}{\xi<\kappa+2}$ of type $(\kappa:2)$ since otherwise $\soast{\alpha_\xi}{\xi<\kappa}\subseteq B_{\alpha_\kappa}\cap B_{\alpha_{\kappa+1}}$, contradicting (2) in the construction of the $B_\alpha$'s.

Suppose there is a $0$-homogeneous set $A$ of size $\kappa^+$. In the construction we ensured that the $B_\alpha$'s met $d_\alpha$ often enough so that the following claim gives a contradiction.
\begin{claim*}\label{c1}
There is $\beta<\kappa^+$ so that $d_\beta\subseteq A$ and $d_\beta$ is not almost contained in any $<\kappa$-union of sets of the form $B_{\alpha'}$ for $\alpha'<\kappa^+$.
\end{claim*}
\begin{proofofclaim}
The proof of the claim splits into two cases depending on the cardinality of $C:=\soast{\alpha}{|B_\alpha\cap A|=\kappa}$.

The first case is that $|C|\le \kappa$. Then $|A\setminus\bigcup_{\alpha\in C} B_\alpha|>\kappa$, so there is some $\beta<\kappa^+$ with $d_\beta\subseteq  A\setminus\bigcup_{\alpha\in C} B_\alpha$. This case is finished by observing that $d_\beta$ has intersection of size $<\kappa$ with every $B_\alpha$.

The second case is that $|C|=\kappa^+$. For each $i<\kappa^+$, define $\xi(i)$ and $\zeta(i)$ so that
\begin{itemize}
\item $e_{\xi(i)}\subseteq \bigcup_{a\in C} B_\alpha\cap A$.
\item $|B_{\zeta(i)}\cap e_{\xi(i)}|= \kappa$, if such exists, and $\zeta(i)$ is undefined otherwise.
\item $e_{\xi(i)}\cap\bigcup_{j<i}(B_{\zeta(j)}\cup e_{\xi(j)})=\emptyset$.
\end{itemize}

There is some $\gamma<\kappa^+$ so that $e_\gamma\subseteq \soast{\xi(i)}{i<\kappa^+}$. Let $\beta$ be so that $d_\beta=\bigcup_{\xi\in e_\gamma} e_\xi$ and let $I=\soast{i<\kappa^+}{\xi(i)\in e_\gamma}$.

Suppose for a contradiction that $d_\beta$ is almost contained in $\bigcup_{\nu<\eta} B_{\alpha_\nu}$ for some $\alpha_\nu$, $\nu<\eta$. There are two subcases.

In the first subcase, there are $\kappa$ many $i\in I$ on which $\zeta(i)$ is defined. For each such $i$, $|B_{\zeta(i)}\cap e_{\xi(i)}|=\kappa$ so there is some $k(i)< \eta$ so that $|B_{\zeta(i)}\cap B_{\alpha_{k(i)}}|=\kappa$, and since the $B_\alpha$ are almost disjoint, $\zeta(i)=\alpha_{k(i)}$. But this is a contradiction since $|I|=\kappa$ and $\zeta$ is injective, but the range of $i\mapsto k(i)$ has size $<\kappa$.

In the second subcase, there are $<\kappa$ many $i\in I$ on which $\zeta(i)$ is defined, so we can take $\xi^*\in e_\gamma$ so that $|e_{\xi^*}\cap B_{\alpha_k}|<\kappa$ for all $k< n$. But this implies that $|e_{\xi^*}|<\kappa$, contradiction.
\end{proofofclaim}

To finish the proof of the theorem, take $\alpha\in A$ larger than $\beta$. By (3) in the construction of the $B_\alpha$'s, there is some $\alpha'$ with $\alpha' \in B_\alpha\cap d_\beta$. Now $\alpha'\in d_\beta\subseteq \beta<\alpha$, so $c(\alpha',\alpha)=1$. Since $d_\beta\subseteq A$, we have $\alpha',\alpha\in A$, contradicting $0$-homogeneity of $A$.
\end{proof}

We now give a few remarks regarding the relationship between the $\stick$ principle and polar\iz ed partition relations. We conjecture that the hypothesis of Theorem \ref{negpolarized} cannot be changed to $\stick=\aleph_1$. Curiously, however, $\stick=\aleph_1$ actually gives a polar\iz ed partition relation involving larger cardinals.
\begin{proposition}
For any regular cardinal $\kappa>\stick$,
\begin{align*}
\binom{\kappa}{\omega_1} \longrightarrow\binom{\kappa}{\omega}^{1,1}_2.
\end{align*}
\end{proposition}
\begin{proof}
Suppose $c:\kappa\times\omega_1\rightarrow 2$ is any col\ou ring. For every $\alpha<\kappa$, $i\in\{0,1\}$, let $S^i_\alpha$ be the set $\soast{\beta<\omega_1}{c(\alpha,\beta)=i}$. For each $\alpha<\kappa$, there is $i(\alpha)$ so that $|S^{i(\alpha)}_\alpha|=\aleph_1$, so there are $B_0\in[\kappa]^\kappa$ and $i^*\in \{0,1\}$ so that if $\alpha\in B_0$, then $i(\alpha)=i^*$.

Let $\langle d_\alpha:\alpha<\stick\rangle$ be a sequence witnessing $\stick$. For every $\alpha\in B_0$, choose $\beta(\alpha)<\omega_1$ so that $d_\beta\subseteq S^{i^*}_\alpha$. Since $\kappa>\stick$ is regular, there is a $B_1\in[B_0]^\kappa$ and $\beta^*<\stick$ so that $d_{\beta^*}\subseteq S^{i^*}_\alpha$ for all $\alpha\in B_1$, and therefore $c\upharpoonright (B_1\times d_{\beta^*})$ is constant.
\end{proof}

\section{Partition Relations opposite the Unbounding Number and the Stick-Principle}

Takahashi proved in \cite{987T1} that one can derive $\omega_1\omega \doesntarrow (\omega_1\omega, 3)^2$ in the system $\ZFC + \mathfrak{d} = \aleph_1 + \stick$. The same can be done in the system $\ZFC + \mathfrak{b} = \aleph_1 + \stick$ as shown in \cite{017LW0}. We now show that the latter system is also sufficient to derive the negative partition relation shown to follow from $\CH$ by Baumgartner and Hajnal in \cite{987BH0}.

\begin{theorem}[\cite{017LW0}]
Suppose that $\kappa$ is regular and $\lambda = \kappa^+ = \mathfrak{b}_\kappa = \stick(\kappa)$. Then $\lambda\kappa \doesntarrow (\lambda\kappa, 3)^2$.
\end{theorem}

\begin{corollary}
If $\kappa$ is regular, $\lambda = \kappa^+ = \mathfrak{b}_\kappa = \stick(\kappa)$, then $\alpha \doesntarrow (\lambda\kappa, 3)^2$ for all $\alpha < \lambda^2$.
\end{corollary}

%\newpage

\begin{theorem}
\label{theorem : quadruples, the unbounding number and stick}
Suppose that $\kappa$ is an infinite regular cardinal and $\lambda = \kappa^+ = \mathfrak{b}_\kappa = \stick(\kappa)$. Then $\lambda^2 \doesntarrow (\lambda\kappa, 4)^2$.
\end{theorem}

\begin{proof}
As in the proof of Theorem \ref{theorem : quadruples and scales}, we begin by fixing
\begin{align}
\text{A bijection } b_\xi & : \xi \longleftrightarrow \kappa, \\
\text{an increasing function } u_\xi & : \kappa \longrightarrow \kappa, \\
\text{a set } s_\xi \in [\lambda]^\kappa.
%\text{, and} \\\text{a bijection } g_\xi : \kappa & \longleftrightarrow \xi \times \xi.
\end{align}
More specifically we do this in a way such that
\begin{align}
\soast{u_\xi}{\xi < \lambda} & \text{ is an unbounded family of functions in } \soafft{\kappa}{\kappa} \text{ such that} \\
& \forall \xi < \lambda\forall \nu < \xi(u_\nu \leqslant^* u_\xi), \\
\soast{s_\xi}{\xi < \lambda} & \text{ witnesses } \stick(\kappa) \text{, \ie\ } \forall X \in [\lambda]^\lambda \exists \xi < \lambda(s_\xi \subseteq X).
%\\ \text{Moreover we choose } & e, f : \lambda \times \kappa \longrightarrow \lambda \text{ such that} \\ &\forall\xi < \lambda\forall\nu < \kappa : g_\xi(\nu) = \opair{e(\xi, \nu)}{f(\xi, \nu)}.
\end{align}
Again, as in the proof of Theorem \ref{theorem : quadruples and scales}, we are going to define a graph $\opair{\lambda^2}{\Psi}$. For ordinals $\alpha, \beta, \gamma, \delta < \lambda$ we will have $\{\lambda\alpha + \beta, \lambda\gamma + \delta\} \in \Psi$ \ifff $\lambda\alpha + \beta \in \Lambda_{\delta, \gamma}$ or $\lambda\gamma + \delta \in \Lambda_{\beta, \alpha}$. Moreover $\lambda\alpha + \beta \in \Lambda_{\delta, \gamma}$ implies $\gamma < \alpha < \beta < \delta$. We also, for every $\gamma \in \beta \setminus (\alpha + 1)$, let $\Delta_{\beta, \alpha, \gamma} \dfeq \soast{\delta \in \beta \setminus \alpha}{\lambda\gamma + \delta \in \Lambda_{\beta, \alpha}}$. In other words, we have
\begin{align}
\Psi = \bigcup_{\beta < \lambda} \bigsoast{\{\lambda\alpha + \beta\} \cup \{\xi\}}{\alpha < \beta \wedge \xi \in \Lambda_{\beta, \alpha}} \\
\text{and } \forall \beta < \lambda \forall \alpha < \beta : \Lambda_{\beta, \alpha} = \bigcup_{\gamma < \beta} \soast{\lambda\gamma + \xi}{\xi \in \Delta_{\beta, \alpha, \gamma}}.
\end{align}
The analogy continues in that $\Psi$ is defined in a four-layered induction. The first layer has length $\lambda$ and in step $\beta < \lambda$ we define $\Lambda_{\beta, \alpha}$ for every $\alpha < \beta$. The second layer has length $\beta$ and in step $\alpha < \beta$ we define $\Lambda_{\beta, \alpha}$. The third layer has length $\kappa$, determining in step $\zeta$ the sets $\Lambda_{\beta, \alpha, \gamma}$ for all $\gamma \in \Gamma_{\beta, \alpha, \zeta}$ where $\Gamma_{\beta, \alpha, \zeta} \in [b_\beta^{-1}(\zeta) + 1\setminus (\beta + 1)]^{<\kappa}$. If $b_\beta(\zeta) \in \bigcup_{\xi < \zeta} \Gamma_{\beta, \alpha, \xi}$, then $\Gamma_{\beta, \alpha, \zeta} \dfeq 0$, otherwise we define the sequence $\seq{\Phi'_{\beta, \alpha, \zeta, n}}{n < \omega}$ inductively by $\Phi'_{\beta, \alpha, \zeta, 0} \dfeq \{b_\beta(\zeta)\}$ and
\begin{align}
\label{b & stick : assignment : Phi's}
\Phi'_{\beta, \alpha, \zeta, n} \dfeq & \bigsoast{\gamma \in \beta \setminus (\alpha + 1)}{\exists \mu \in \bigcup_{\xi < \zeta} \Gamma_{\beta, \alpha, \xi} \cup \bigcup_{k < n} \Phi'_{\beta, \alpha, \zeta, k} \setminus \gamma \big(b_\mu(\gamma) < b_\mu(\alpha) \\
& \mspace{80mu} \vee \exists \rho \in \bigcup_{\xi < \zeta} \Gamma_{\beta, \alpha, \xi} \cup \bigcup_{k < n} \Phi'_{\beta, \alpha, \zeta, k} \cap \mu \setminus \gamma(b_\mu(\gamma) < b_\mu(\rho)\big)}, \\
\text{moreover we set} \notag\\
\Phi_{\beta, \alpha, \zeta} \dfeq & \bigcup_{n < \omega} \Phi'_{\beta, \alpha, \zeta, n} \text{, and} \\
\Gamma_{\beta, \alpha, \zeta} \dfeq & \Phi_{\beta, \alpha, \zeta} \setminus \bigcup_{\xi < \zeta} \Gamma_{\beta, \alpha, \xi}.
\end{align}
Note that
\begin{align}
\label{b & stick : assertion : size of Theta's}
\card{\Phi_{\beta, \alpha, \zeta}} < \kappa.
\end{align}
First of all, it is easy to prove by induction that for all $n < \omega$ the set $\Phi'_{\beta, \alpha, \zeta, n}$ has cardinality less than $\kappa$. It is comparably easy to see that the natural numbers $k$ for which $\Phi'_{\beta, \alpha, \zeta, k}$ is nonempty, form an initial segment of $\omega$.  If $\kappa$ is uncountable, then, as $\kappa$ is regular, clearly $\card{\Phi_{\beta, \alpha, \zeta}} < \kappa$. So assume towards a contradiction that \eqref{b & stick : assertion : size of Theta's} fails. We have $\kappa = \omega$ and for every natural number $k$ the set $\Phi'_{\beta, \alpha, \zeta, k}$ is finite. But then, by \eqref{b & stick : assignment : Phi's}, the sequence $\seq{\max(\Phi'_{\beta, \alpha, \zeta, k + 1} \setminus \bigcup_{m \leqslant k} \Phi'_{\beta, \alpha, \zeta, m})}{k < \omega}$ is a descending one of ordinals, a contradiction. Therefore \eqref{b & stick : assertion : size of Theta's}.

Now let $\seq{\gamma_{\beta, \alpha, \zeta, \nu}}{\nu < \vartheta_{\beta, \alpha, \zeta}}$ be the increasing enumeration of $\Gamma_{\beta, \alpha, \zeta}$. Now inductively for every $\nu < \vartheta_{\beta, \alpha, \zeta}$, let
\begin{align}
\label{b & stick : assignment : Xi's}
\Xi_{\beta, \alpha, \zeta, \nu} & \dfeq \bigcup\bigsoast{\Delta_{\delta, \gamma_{\beta, \alpha, \xi, \mu}, \gamma_{\beta, \alpha, \zeta, \nu}}}{\big((\xi < \zeta \wedge \mu < \vartheta_{\beta, \alpha, \xi}) \vee (\xi = \zeta \wedge \mu < \nu)\big) \wedge \gamma_{\beta, \alpha, \xi, \mu} < \gamma_{\beta, \alpha, \zeta, \nu} \wedge \delta \in \Delta_{\beta, \alpha, \gamma_{\beta, \alpha, \xi, \mu}}} \\
\text{and} & \notag\\
\label{b & stick : assignment : Delta's}
\Delta_{\beta, \alpha, \gamma_{\beta, \alpha, \zeta, \nu}} & \dfeq \Bigsoast{\min\big(s_\xi \setminus \Xi_{\beta, \alpha, \zeta, \nu})}{\xi < \beta \wedge s_\xi \subseteq \beta \wedge \exists \rho \leqslant \zeta\big(b_{b_\beta^{-1}(\rho)}(\xi) < u_\beta(b_{b_\beta^{-1}(\rho)}(\gamma_{\beta, \alpha, \zeta, \nu}))\big)}.
\end{align}

This finishes the induction. Now we set $\Lambda_{\beta, \alpha} \dfeq \soast{\lambda\gamma_{\beta, \alpha, \zeta, \nu} + \delta}{\zeta < \kappa \wedge \nu < \vartheta_\zeta \wedge \delta \in \Delta_{\beta, \alpha, \gamma_{\beta, \alpha, \zeta, \nu}}}$ which finishes the definition of $\Psi$.

Now we are going to check that this provides what was demanded. So let $Q \in [\lambda^2]^4$. Let $Q = \soast{\lambda \alpha_k + \beta_k}{k \in 5 \setminus 1}$.
Without loss of generality we have $\alpha_1 < \alpha_2 < \alpha_3 < \alpha_4$. Assume towards a contradiction, that $[Q]^2 \subseteq \Psi$. It follows that
\begin{align}
\alpha_1 < \alpha_2 < \alpha_3 < \alpha_4 < \beta_4 < \beta_3 < \beta_2 < \beta_1 \text{ and } \soast{\lambda\alpha_k + \beta_k}{k \in 5 \setminus 2} \in [ \Lambda_{\beta_1, \alpha_1}]^3.
\end{align}

For $i \in \{1, 2\}$ and $j \in 5 \setminus (i + 1)$, let $\zeta_{ij}$ be the induction step in which $\lambda\alpha_j + \beta_j$ was added to $\Lambda_{\beta_i, \alpha_i}$ and let $\nu_{ij} < \vartheta_{\zeta_{ij}}$ be such that $\alpha_j = \gamma_{\beta_i, \alpha_i, \zeta_{ij}, \nu_{ij}}$.
Now we distinguish two cases:

First assume that there is an $i \in \{1, 2\}$ and a $k \in 4 \setminus (i + 1)$ such that $\zeta_{ik} \leqslant \zeta_{i(k + 1)}$. Then, either $\zeta_{ik} < \zeta_{i(k + 1)}$ or both $\zeta_{ik} = \zeta_{i(k + 1)}$ and $\nu_{ik} < \nu_{i(k + 1)}$.
As $\{\lambda\alpha_i + \beta_i, \lambda\alpha_k + \beta_k\} \in \Psi$ and $\alpha_i < \alpha_k$, we have $\lambda\alpha_k + \beta_k \in \Lambda_{\beta_i, \alpha_i}$. As $\beta_k = \gamma_{\beta_i, \alpha_i, \zeta_{ik}, \nu_{ik}}$, this implies  $\beta_k \in \Delta_{\beta_i, \alpha_i, \alpha_k}$. Since $\alpha_k = \gamma_{\beta_i, \alpha_i, \zeta_{ik}, \nu_{ik}} < \gamma_{\beta_i, \alpha_i, \zeta_{i(k + 1)}, \nu_{i(k + 1)}} = \alpha_{k + 1}$, \eqref{b & stick : assignment : Xi's} implies
\begin{align}
& \Delta_{\beta_k, \alpha_k, \alpha_{k + 1}} \subseteq \Xi_{\alpha_i, \beta_i, \zeta_{i(k + 1)}, \nu_{i(k + 1)}}. \\
\text{As } \{\lambda\alpha_k + \beta_k, \lambda\alpha_{k + 1} + \beta_{k + 1}\} \in \Psi & \text{, we have } \lambda\alpha_{k + 1} + \beta_{k + 1} \in \Lambda_{\beta_k, \alpha_k}. \\
\text{Therefore, } \beta_{k + 1} \in & \Delta_{\beta_k, \alpha_k, \alpha_{k + 1}}.
\end{align}
So $\beta_{k + 1} \in \Xi_{\alpha_i, \beta_i, \zeta_{i(k + 1)}, \nu_{i(k + 1)}}$ and by \eqref{b & stick : assignment : Xi's}, we get $\beta_{k + 1} \notin \Delta_{\beta_i, \alpha_i, \alpha_{k + 1}}$. This, however, implies $\lambda\alpha_{k + 1} + \beta_{k + 1} \notin \Lambda_{\beta_i, \alpha_i}$ which means that $\{\lambda\alpha_i + \beta_i, \lambda\alpha_{k + 1} + \beta_{k + 1}\} \in [Q]^2 \setminus \Psi$, a contradiction.

Now assume that this is not the case, \ie\ $\zeta_4 < \zeta_3 < \zeta_2$ and $\zeta_8 < \zeta_6$. Let $\ell$ be the least natural number for which $\alpha_3 \in \Phi'_{\beta_1, \alpha_1, \zeta_3, \ell}$. Then,
\begin{align}
\alpha_2 \notin \bigcup_{\zeta < \zeta_2} \Gamma_{\beta_1, \alpha_1, \zeta} \supseteq \bigcup_{\zeta \leqslant \zeta_3} \Gamma_{\beta_1, \alpha_1, \zeta} \supseteq \Phi_{\beta_1, \alpha_1, \zeta_3} \supseteq \Phi'_{\beta_1, \alpha_1, \zeta_3, \ell + 1}, \\
\text{so by } \eqref{b & stick : assignment : Phi's}, \notag\\
\forall \mu \in \bigcup_{\xi < \zeta_3} \Gamma_{\beta_1, \alpha_1, \xi} \cup \bigcup_{k \leqslant \ell} \Phi'_{\beta_1, \alpha_1, \zeta_3, k} \setminus \alpha_2\big( b_\mu(\alpha_1) < b_\mu(\alpha_2) \wedge (\forall \rho \in \bigcup_{\xi < \zeta_3} \Gamma_{\beta_1, \alpha_1, \xi} \cup \bigcup_{k \leqslant \ell} \Phi'_{\beta_1, \alpha_1, \zeta_3, k} \cap \mu \setminus \alpha_2(b_\mu(\rho) \leqslant b_\mu(\alpha_2))\big).
\end{align}
As $b_{\alpha_4}$ is injective, $\alpha_3 \in \Gamma_{\beta_1, \alpha_1, \zeta_4}$ and $\alpha_3 \in \Phi'_{\beta_1, \alpha_1, \zeta_3, \ell}$, we have
\begin{align}
\label{b & stick : inequality : bijections}
b_{\alpha_4}(\alpha_3) < b_{\alpha_4}(\alpha_2).
\end{align}

Now let $m$ be the least natural number such that $\alpha_4 \in \Phi'_{\beta_2, \alpha_2, \zeta_8, m}$. Then,
\begin{align}
\alpha_3 \notin \bigcup_{\zeta < \zeta_6} \Gamma_{\beta_2, \alpha_2, \zeta} \supseteq \bigcup_{\zeta \leqslant \zeta_8} \Gamma_{\beta_2, \alpha_2, \zeta} \supseteq \Phi_{\beta_2, \alpha_2, \zeta_8} \supseteq \Phi'_{\beta_2, \alpha_2, \zeta_8, m + 1}, \\
\text{so by } \eqref{b & stick : assignment : Phi's}, \notag\\
\forall \mu \in \bigcup_{\xi < \zeta_8} \Gamma_{\beta_2, \alpha_2, \xi} \cup \bigcup_{k \leqslant m} \Phi'_{\beta_2, \alpha_2, \zeta_8, k} \setminus \alpha_3\big( b_\mu(\alpha_2) < b_\mu(\alpha_3) \wedge (\forall \rho \in \bigcup_{\xi < \zeta_8} \Gamma_{\beta_2, \alpha_2, \xi} \cup \bigcup_{k \leqslant m} \Phi'_{\beta_2, \alpha_2, \zeta_8, k} \cap \mu \setminus \alpha_3(b_\mu(\rho) \leqslant b_\mu(\alpha_3))\big).
\end{align}
As $\alpha_4 \in \Phi'_{\beta_1, \alpha_1, \zeta_8, m} \setminus \alpha_3$, we get $b_{\alpha_4}(\alpha_2) < b_{\alpha_4}(\alpha_3)$, contradicting \eqref{b & stick : inequality : bijections}.

Now let %assume towards a contradiction that there is an
$H \in [\lambda^2]^{\lambda\kappa}$. %such that $[H]^2 \subseteq [\lambda^2]^2 \setminus \Psi$.
We define
\begin{align}
A \dfeq \soast{\mu < \lambda}{\card{H \cap \lambda(\mu + 1) \setminus \lambda\mu} = \lambda}.
\end{align}
Let $\alpha \dfeq \min(A)$ and $\xi \dfeq \sup(A)$ and $\tau \dfeq \min\soast{\mu \in \lambda \setminus \xi}{\forall \nu \in A \exists \rho < \mu \forall \sigma \in s_\rho : \lambda\nu + \sigma \in H}$. Consider the functions
\begin{align}
\label{b & stick : assignment : f}
f : \kappa & \longrightarrow \kappa, \\
\mu & \longmapsto \min\soast{\nu \in \kappa \setminus \mu}{b_\tau^{-1}(\nu) \in A} \text{ and} \\
\label{b & stick : assignment : g}
g : \kappa & \longrightarrow \kappa, \\
& \mu \longmapsto \min\soast{\nu < \kappa}{\forall \rho \leqslant \mu\big(b_\tau^{-1}(\rho) \in A \rightarrow \exists \sigma < \nu \forall \varphi \in s_{b_\tau^{-1}(\sigma)} : \lambda b_\tau^{-1}(\rho) + \varphi \in H\big)}.
\end{align}
Now set $h \dfeq g \circ f$ and let $\beta \in \lambda \setminus \tau$ be such that $\lambda\alpha + \beta \in H$ and $u_\beta$ is unbounded over $h$. Let $\iota \dfeq b_\beta(\tau)$ and $\psi \in \kappa \setminus \max\big(\iota, b_\beta(\alpha) + 1\big)$ be such that $u_\beta(\psi) > h(\psi)$ and let $\gamma \dfeq b_\tau^{-1}\big(\min\soast{\mu \in \kappa \setminus \psi}{b_\tau(\mu) \in A}\big)$. By \eqref{b & stick : assignment : f}, we have $f(\psi) \geqslant b_\tau(\gamma)$. Now by \eqref{b & stick : assignment : g},
there is a $\sigma < h(\psi)$ such that for all $\varphi \in s_{b_\tau^{-1}(\sigma)}$ we have $\lambda \gamma + \varphi \in H$. Let $\zeta < \kappa$ and $\nu < \vartheta_{\beta, \alpha, \zeta}$ be such that $\gamma = \gamma_{\beta, \alpha, \zeta, \nu}$. Then, as $\card{\Xi_{\beta, \alpha, \zeta, \nu}} < \kappa$, there is a $\delta \in s_{b_\tau^{-1}(\sigma)} \setminus \Xi_{\beta, \alpha, \zeta, \nu} \subseteq \Delta_{\beta, \alpha, \gamma}$. Then $\lambda \gamma + \delta \in \Lambda_{\beta, \alpha}$ and, consequently, $\{\lambda \alpha + \beta, \lambda \gamma + \delta\} \in \Psi$.
\end{proof}

\begin{corollary}
If $\kappa$ is regular and $\lambda = \kappa^+ = \mathfrak{b}_\kappa = \stick(\kappa)$, then $\alpha \doesntarrow (\lambda\kappa, 4)^2$ for all $\alpha < \lambda^2\kappa$.
\end{corollary}

\begin{corollary}
If $\mathfrak{b} = \aleph_1 = \stick$ then $\alpha \doesntarrow (\omega_1\omega, 4)^2$ for all $\alpha < \omega_1^2\omega$.
\end{corollary}

\section{Baire category and a col\ou ring of Shelah}

Answering a question of Erd\H{o}s and Hajnal, Shelah \cite{975S2} proved that $\CH$ implies that there is a col\ou ring proving $\omega_1\doesntarrow[\omega_1]^2_\omega$ with no triangle with edges of three different col\ou rs. Using a similar col\ou ring, we reduce this hypothesis to the existence of a Luzin set, i.e., an uncountable subset of $\soafft{\omega}{2}$ so that each of its uncountable subsets is nonmeag\re.

Abstractly, some Luzin-type properties are relevant to proving negative partition relations, since they capture properties of sets that are inherited by all large subsets. In fact, this is the reason that assumptions like $\mathfrak{b}=\aleph_1$ or the existence of a scale are so prominent in the theorems in this paper.

Let us place Luzin sets in the picture of cardinal characteristics. The existence of a Luzin set follows from $\CH$, \cf\ \cite{914L0} (and in fact from $\cof(\meagre)=\aleph_1$). Judah and Shelah proved that it is consistent that $\non(\meagre)=\aleph_1$ and there are no Luzin sets, \cf\ \cite{994JS0}. The existence of a Luzin set does not entail $\mathfrak{d}=\aleph_1$: starting from a model of $\CH$, add $\aleph_2$-many Cohen reals. Each uncountable subset of the Cohen reals is nonmeag\re, since it cannot be a subset of any Borel meag\re\ set (which lives in an intermediate extension of just countably many of the Cohen factors). However, $\mathfrak{d}=\aleph_2$ in this model.

\begin{theorem}
Suppose there exists a Luzin set. Then there is a col\ou ring proving $\omega_1\doesntarrow[\omega_1]^2_\omega$ with no triangle with edges of three different col\ou rs.
\end{theorem}
\begin{proof}
Fix a Luzin set $\langle f_\alpha\mid\alpha<\omega_1\rangle$ of size $\aleph_1$. For $\alpha<\beta<\omega_1$, define $\Delta(\alpha,\beta)$ to be the least $n<\omega$ such that $f_\alpha(n)\neq f_\beta(n)$.

Let $s:\omega\rightarrow\omega$ be a function so that for every $i<\omega$, $s^{-1}[i]$ is infinite. Define the col\ou ring $c:[\omega_1]^2\rightarrow \omega$ to be $c(\alpha,\beta)\dfeq s(\Delta(\alpha,\beta))$.

There are no triangles which get three col\ou rs from $c$.
\begin{claim*}
For distinct $\alpha,\beta,\gamma<\omega_1$, the smaller two among $\Delta(\alpha,\beta),\Delta(\beta,\gamma),\Delta(\gamma,\alpha)$ are equal.
\end{claim*}
\begin{proofofclaim}
Towards a contradiction, suppose $\Delta(\alpha,\gamma)<\Delta(\beta,\gamma)\leqslant D(\alpha,\beta)$. By definition of $\Delta$,
$$f_\alpha(\Delta(\alpha,\gamma))\neq f_\gamma(\Delta(\alpha,\gamma)),$$
but since $\Delta(\alpha,\gamma)<\Delta(\beta,\gamma)$,
$$f_\gamma(\Delta(\alpha,\gamma))=f_\beta(\Delta(\alpha,\gamma)).$$
In total,
$$f_\alpha(\Delta(\alpha,\gamma))\neq f_\beta(\Delta(\alpha,\gamma))$$
but this contradicts $\Delta(\alpha,\gamma)< D(\alpha,\beta)$.
\end{proofofclaim}

Now we show that $c$ satisfies $\omega_1\doesntarrow[\omega_1]^2_\omega$.

\begin{claim*}
If $X\subseteq \soafft{\omega}{2}$ is nonmeag\re, then for all but finitely many $n$, there exists some $t\in \soafft{n}{2}$ and $f,g\in X$ so that $t^\frown \langle 0\rangle\subseteq f$ and $t^\frown \langle 0\rangle\subseteq g$.
\end{claim*}
\begin{proofofclaim}
Suppose not. Then there are infinitely many $n$ so that for every $f\in X$, $f(n)=i_{f\upharpoonright n}$ for some $i_{f\upharpoonright n}$ which depends only on $f\upharpoonright n$. The closure $\mathrm{cl}(X)$ then must also have this property, which implies that the interior of $\mathrm{cl}(X)$ is empty, and hence $X$ is nowhere dense.
\end{proofofclaim}

Suppose $A\subseteq \omega_1$. Then by the Luzin property, $\soast{f_\alpha}{\alpha\in A}$ is nonmeag\re, so there is some $n^*$ large enough so that all $n\geqslant n^*$ satisfy the condition in the last claim. Let $m<\omega$ be arbitrary. Choose $n\geqslant n^*$ so that $s(n)=m$ and $\alpha,\beta\in A$ so that $f_\alpha\upharpoonright n=f_\beta\upharpoonright n$ but $f_\alpha(n)=0$ and $f_\beta(n)=1$. Then $\Delta(\alpha,\beta)=m$, and $c(\alpha,\beta)=s(\Delta(\alpha,\beta))=m$.
\end{proof}

\begin{corollary}
If $\cof(\meagre)=\aleph_1$, then there is a col\ou ring proving $\omega_1\doesntarrow[\omega_1]^2_\omega$ with no triangle with edges of three different col\ou rs.
\end{corollary}

\begin{proof}
It suffices to show that there is a Luzin set under these hypotheses. We give the argument of this classical fact here.

Let $\seq{A_\alpha}{\alpha<\omega_1}$ be a cofinal family of meag\re\ sets. Let $B_\alpha=\bigcup_{\beta\leqslant\alpha}A_\beta$, so $\seq{B_\alpha}{\alpha<\omega_1}$ is a $\subseteq$-increasing cofinal family of meag\re\ sets. Thin out so that $B_\alpha\neq B_\beta$ if $\alpha\neq \beta$. For each $\alpha<\omega_1$, pick $x_\alpha\in B_{\alpha+1}\setminus B_\alpha$. Then $\soast{x_\alpha}{\alpha<\omega_1}$ is a Luzin set: none of its uncountable subsets are contained in any single $B_\alpha$.
\end{proof}

\section{Questions}
This paper leaves open many natural questions. On the relationship between cardinal invariants and $\stick$, even the following is not known:
\begin{question}
Does $\stick=\aleph_1$ imply $\mathfrak{b}=\aleph_1$?
\end{question}

We suspect that the assumption $\mathfrak{d}=\aleph_1$ appearing in the theorems above and in Larson \cite{998L0} can be reduced somewhat. As a starting point, we ask:
\begin{question}
Does $\stick = \aleph_1$ imply $\omega_1\omega^2 \doesntarrow (\omega_1\omega^2, 3)^2$ or even $\omega_1\omega \doesntarrow (\omega_1\omega, 3)^2$? What about $\mathfrak{b}=\aleph_1$?
\end{question}
\begin{question}
Does $\stick = \aleph_1$ imply $\omega_1^2 \doesntarrow (\omega_1\omega, 4)^2$? What about $\mathfrak{b}=\omega_1$?
\end{question}
The positive partition relations are known to follow from $\MA + \neg\CH$ (\cf\ Section \ref{section : introduction}) and can be forced with by iterating c.c.c. posets which add homogeneous sets to counterexamples using finite conditions. However, these posets destroy $\stick = \aleph_1$ and $\mathfrak{b}=\aleph_1$.

Perhaps easier would be to consider the polar\iz ed partition relation, where our basic question is:
\begin{question}
Does $\stick = \aleph_1$ imply $\binom{\omega_1}{\omega}\doesntarrow\binom{\omega_1}{\omega}^{1,1}_2$?
\end{question}

We would also be interested in any relationships of these relations with the category invariants.

Whether $\omega_1^2 \longrightarrow (\omega_1^2, 3)^2$ is consistent is a long-standing open question of Baumgartner. In our context, we would be interested in reducing the hypotheses used by Larson ($\mathfrak{d}=\aleph_1$) and Weinert--Lambie-Hanson ($\mathfrak{b}=\stick=\aleph_1$) to prove the negative relation. A natural mutual strengthening of these results would be:
\begin{question}
Does $\mathfrak{b} = \aleph_1$ imply $\omega_1^2 \doesntarrow (\omega_1^2, 3)^2$?
\end{question}

The arguments here from cardinal characteristics roughly calibrate the strength of negative partition relations. We can ask if there is a direct implication.
\begin{question}
Does $\omega_1^2\doesntarrow(\omega_1\omega,4)$ imply $\omega_1\omega\doesntarrow(\omega_1\omega,3)^2$?
\end{question}

\begin{question}
Does $\omega_1^2\doesntarrow(\omega_1\omega,4)\wedge \omega_1^2\doesntarrow(\omega_1^2,3)$ imply $\omega_1\omega\doesntarrow(\omega_1\omega,3)^2$?
\end{question}

In this paper, we focused on $\omega_1$ for the value of the cardinal characteristic in the hypothesis and for the resources in the partition relations. However, it may be possible that this is not necessary.

\begin{question}
Does $\mathfrak{b} \leqslant \stick$ hold by way of $\ZFC$? For which cardinals $\kappa$ does $\lambda = \kappa^+$ and $\stick(\kappa) = \lambda$ imply the existence of an unbounded family of cardinality $\soafft{\lambda}{\lambda}$?
\end{question}

\begin{question}
Is $\min(\cov(\meagre), \cov(\mzero)) \leqslant \stick$?
\end{question}

%\newpage

\nocite{014BR0}
\nocite{012GS0}

\bibliography{thilo}
\bibliographystyle{thilo3}
\end{document}
}}